\documentclass[dvips,preprint]{imsart}

\RequirePackage[OT1]{fontenc} \RequirePackage{amsthm,amsmath}
\RequirePackage[numbers]{natbib}
\RequirePackage[colorlinks,citecolor=blue,urlcolor=blue]{hyperref}
\RequirePackage{hypernat}
\usepackage{amsmath}
\usepackage{amssymb}
\usepackage{amsmath}
\usepackage{amsfonts}
\usepackage[american]{babel}
\usepackage{graphicx}
\usepackage{flafter}
\usepackage[section]{placeins}
\arxiv{math.PR/0000000}

\startlocaldefs

\def\P{{\mathbb P}}
\def\E{{\mathbb E }}

\def\Bbb E{\mathbb{E}}
\def\Bbb R{\mathbb{R}}
\newtheorem{proposition}{Proposition}
\newtheorem{lemma}{Lemma}
\newtheorem{theorem}{Theorem}
\newtheorem{corollary}{Corollary}

\makeatletter \@addtoreset{equation}{section}

\makeatother

\font\tencmmib=cmmib10 \skewchar\tencmmib '60
\newfam\cmmibfam
\textfont\cmmibfam=\tencmmib

\font\tenmsb=msbm10 
\def\Bbb#1{\hbox{\tenmsb#1}}

\def\lessim{\ \lower4pt\hbox{$
\buildrel{\displaystyle <}\over\sim$}\ }
\def\gessim{\ \lower4pt\hbox{$\buildrel{\displaystyle >}
\over\sim$}\ }

\def\AA{{\cal A} }

\def\go0{\to 0}

\def\leftitem#1{\item{\hbox to\parindent{\enspace#1\hfill}}}

\def\sg{\sigma}

\def\sg2{\sigma^2}

\def\__{_{\infty}}

\numberwithin{equation}{section} \theoremstyle{plain}

\newcommand{\1}{{\rm 1}\kern-0.24em{\rm I}}
\def\E{\mathbb E}
\def\R{\mathbb R}
\newtheorem{assumption}{Assumption}

\newtheorem{defi}{Definition}
\endlocaldefs

\begin{document}

\begin{frontmatter}
\title{High-dimensional covariance matrix estimation with missing observations} \runtitle{Covariance matrix estimation with missing observations}

\begin{aug}
\author{\fnms{Karim} \snm{Lounici}\thanksref{m1}\ead[label=e2]{klounici@math.gatech.edu}}
\thankstext{m1}{Supported in part by NSF Grant DMS-11-06644 and Simons foundation Grant 209842}
\runauthor{K. Lounici}

\affiliation{Georgia Institute of Technology\thanksmark{m1}}

\address{School of Mathematics\\
Georgia Institute of Technology\\
Atlanta, GA 30332-0160\\
\printead{e2}\\
}
\end{aug}

\begin{abstract}\
In this paper, we study the problem of high-dimensional approximately low-rank
covariance matrix estimation with missing observations. We propose a
simple procedure computationally tractable in high-dimension and
that does not require imputation of the missing data. We establish
non-asymptotic sparsity oracle inequalities for the estimation of
the covariance matrix with the Frobenius and spectral norms, valid
for any setting of the sample size and the dimension of the
observations. We further establish minimax lower bounds showing that
our rates are minimax optimal up to a logarithmic factor.
\end{abstract}



\end{frontmatter}

\section{Introduction}

Let $X,X_1,\ldots,X_n \in \R^p$ be i.i.d. zero mean vectors with
unknown covariance matrix $\Sigma = \E X \otimes X$. Our objective
is to estimate the unknown covariance matrix $\Sigma$ when the
vectors $X_1,\ldots,X_n$ are partially observed, that is, when some
of their components are not observed. More precisely, we consider
the following framework. Denote by $X^{(j)}_i$ the $j$\emph{-th}
component of the vector $X_i$. We assume that each component
$X^{(j)}_i$ is observed independently of the others with probability
$\delta\in (0,1]$. Note that $\delta$ can be easily estimated by the
proportion of observed entries. Therefore, we will assume in this
paper that $\delta$ is known. Note also that the case $\delta=1$
corresponds to the standard case of fully observed vectors. Let $(\delta_{i,j})_{1\leq i \leq n,1\leq j \leq p}$ be a
sequence of i.i.d. Bernoulli random variables with parameter
$\delta$ and independent from $X_1,\ldots,X_n$. We observe $n$
i.i.d. random vectors $Y_1,\ldots,Y_n\in \R^p$ whose components
satisfy
\begin{equation}\label{equationY}
Y_i^{(j)} = \delta_{i,j}X_i^{(j)},\quad 1\leq i \leq n,\,1\leq j
\leq p.
\end{equation}
We can think of the $\delta_{i,j}$ as masked variables. If
$\delta_{i,j}=0$, then we cannot observe the $j$\emph{-th} component
of $X_i$  and the default value $0$ is assigned to $Y_i^{(j)}$. Our goal is then to estimate $\Sigma$ given the partial
observations $Y_1,\ldots,Y_n$.

The statistical problem of covariance estimation with missing
observations is fundamental in multivariate statistics since it is
often used as the first step to retrieve information in numerous
applications where datasets with missing observations are common:
\begin{enumerate}
  \item Climate studies: $n$ is the number of time points and $p$ the number of observations stations, which may sometimes fail to produce an observation due to break down of measure instruments. As a consequence, the generated datasets usually contain missing
  values.
  \item Gene expression micro-arrays: $n$ is the number of measurements and $p$ the number of tested genes. Despite the improvement of genes expression techniques, the generated datasets frequently contain missing values with up to $90\%$ of genes affected.
  \item Cosmology: $n$ is the number of images produced by a telescope and $p$ is the number of pixels per image. With the development of very large telescopes and wide sky surveys, the generated datasets are huge but usually contain missing observations due to partial sky coverage or defective pixel.
\end{enumerate}

One simple strategy to deal with missing data is to exclude from the
analysis any variable for which observations are missing, thus
restricting the analysis to a subset of fully observed variables. In
gene expression data where $90\%$ of the genes are affected by
missing values, we would be left with too few variables so that the
legitimacy of the statistical analysis becomes questionable. Also,
discarding variables with very few missing observations is a waste
of available information. Existing procedures involve complex
imputation techniques to fill in the missing values through
computationally intensive implementation of the EM algorithm, see
\cite{schneider} and the references cited therein for more details.
In this paper, we propose a simple procedure computationally
tractable in high-dimension that does not require to imput missing
observations or to discard any available observation to recover the
covariance matrix $\Sigma$.

Contemporary datasets are huge with both large sample size $n$ and
dimension $p$ and typically $p\gg n$. Consequently, a question of
considerable practical interest is to perform dimension reduction,
that is finding a good low-dimensional approximation for these huge
datasets. This recent paradigm where high-dimensional objects of
interest admit in fact a small intrinsic dimension has produced
spectacular results in several fields, such as compressed sensing
where it is possible to recover $s$-sparse vectors of dimension $p$
with only $n = \mathcal{O}\left(s \log(ep/s)\right)$ measurements
provided these measurements are carried out properly, see
\cite{BRT09,CT07,DT05,Kolstflour,kol_dantz,L08} and the references
cited therein for more details. An analogous result holds in matrix
completion where recovery of a low-rank matrix $A\in \R^{p\times p}$ via nuclear norm minimization
is possible with as few as $\mathcal{O}\left(pr\log^2 p\right)$
observed entries where $r$ is the rank of $A$, provided the matrix
of interest $A$ satisfies some incoherence condition, see
\cite{cp09,Candes_Recht,Candes_Tao,Gross-2,Keshavan,KLT,Kol10,RFP10,rohde}
for more details. See also \cite{BSW11,Klo11} for rank minimization approach. A popular dimension reduction technique for
covariance matrices is Principal Component Analysis (PCA), which
exploits the spectrum of the sample covariance matrix. In the
high-dimensional setting, \cite{J01} showed that the standard PCA
procedure is bound to fail since the sample covariance spectrum is
too spread out.

Several alternatives have been studied in the literature to provide
better estimates of the covariance matrix in the high-dimensional
setting. A popular approach in Gaussian graphical models consists in
estimating the inverse of the covariance matrix (called
concentration matrix) since it admits a naturally sparse (or approximately sparse) structure
if the dependence graph is itself sparse. See \cite{BGA08,FHT08,MB06,RWRY11,YL07,CLL11} and the references cited therein for more details. A limitation of this approach is that it does not apply to low rank matrices $\Sigma$ since the concentration matrix does not exist in this case. 
An other popular approach assumes that the unknown covariance matrix is sparse,that is most of the entries are
exactly or approximately zero and then proposes to perform either
entrywise thresholding or tapering of the sample covariance matrix
\cite{BL08,CL11,CZZ10,Karoui08b,RBLZ08,RLZ09}. Note that the
sparsity notion adopted in this approach is not adapted to strongly
correlated datasets with dense covariance matrix.

In random matrix theory, an important line of work,
\cite{Karoui08a,J01,JM11} and the references cited therein studied
the asymptotic distribution of the sample covariance matrix
eigenvalues for different settings of $n$ and $p$. See also
\cite{vershynin} for a very nice survey of existing non-asymptotic
results on the spectral norm deviation of the sample covariance
matrix from its population counterpart. In this paper, we adopt this
approach and we will provide further details as we present our
results.

Note that the results derived in the works cited above do not cover
datasets with missing observations. For instance, when the data
contains no missing observation ($\delta=1$), \cite{vershynin}
established a non-asymptotic control on the stochastic deviation
$\|\Sigma_n - \Sigma\|_\infty$ of the empirical covariance matrix
$\Sigma_n = \frac{1}{n}\sum_{i=1}^n X_i \otimes X_i$ provided some
tails conditions are satisfied by the common distribution of
$X_1,\ldots,X_n$. Exploiting these results, it is possible to
establish oracle inequalities for the covariance version of the
matrix Lasso estimator
\begin{equation}\label{matrixLasso}
\hat \Sigma^L = \mathrm{argmin}_{S\in \mathcal{S}_p}\|\Sigma_n -
S\|_2^2+\lambda \|S\|_1,
\end{equation}
where $\mathcal{S}_p$ is the set of $p\times p$
positive-semidefinite symmetric matrices, $\|S\|_2$ and $\|S\|_1$
are respectively the Frobenius and nuclear norm of $S$ and
$\lambda>0$ is a regularization parameter that should be chosen of
the order of magnitude of $\|\Sigma_n - \Sigma\|_\infty$. This
estimator is the covariance version of the matrix Lasso estimator
initially introduced in the matrix regression framework, see
\cite{KLT,rohde} and the references cited therein. To the best of
our knowledge, the procedure (\ref{matrixLasso}) has not been
studied in the covariance estimation problem.

When the data contains missing observations ($\delta<1$), we no longer
have access to $\Sigma_n$. Given the observations $Y_1,\ldots,Y_n$,
we can build the following empirical covariance matrix
$$
\Sigma_n^{(\delta)} = \frac{1}{n}\sum_{i=1}^n Y_i \otimes Y_i.
$$
In this case, a naive approach to derive oracle inequalities consists in computing the matrix Lasso estimator
(\ref{matrixLasso}) with $\Sigma_n$ replaced by
$\Sigma_n^{(\delta)}$. Unfortunately this approach is bound to fail
since $\Sigma_n^{(\delta)}$ is not a good estimator of $\Sigma$ when
$\delta<1$. Indeed, some elementary algebra gives that $\mathbb
E\left( \Sigma_n^{(\delta)}\right) = \Sigma^{(\delta)}$ with
\begin{equation*}
\Sigma^{(\delta)} = (\delta-\delta^2) \mathrm{diag}(\Sigma) +
\delta^2 \Sigma,
\end{equation*}
where $\mathrm{diag}(\Sigma)$ is the $p\times p$ diagonal matrix
obtained by putting all the non-diagonal entries of $\Sigma$ to
zero. When $\delta=1$, we see that $\Sigma^{(1)}=\Sigma$ and
$\Sigma_n^{(1)}=\Sigma_n$. However, when observations are missing
($\delta<1$), $\Sigma^{(\delta)}$ can be very far from $\Sigma$.
Hence, $\Sigma_{n}^{(\delta)}$ will be a poor estimator of $\Sigma$
since it concentrates around its mean $\Sigma^{(\delta)}$ under
suitable tail conditions on the distribution of $X$. Consequently,
the stochastic deviation $\|\Sigma_n^{(\delta)}-\Sigma\|_\infty$
will be too large and the matrix Lasso estimator (\ref{matrixLasso})
with $\Sigma_n$ replaced by $\Sigma_n^{(\delta)}$, which requires
$\lambda$ to be of the order of magnitude of
$\|\Sigma_n^{(\delta)}-\Sigma\|_\infty$, will perform poorly since
its rate of estimation grows with $\lambda$.

We present now our reconstruction procedure based on the following
simple observation
\begin{equation}\label{Sigmarecons}
\Sigma = (\delta^{-1} -
\delta^{-2})\mathrm{diag}\left(\Sigma^{(\delta)} \right) +
\delta^{-2}\Sigma^{(\delta)},\quad \forall 0<\delta\leq 1.
\end{equation}
Therefore, we can define the following unbiased estimator of $\Sigma
$ when the data set contains missing observations
\begin{equation}\label{Sigmaemprecons}
\tilde\Sigma_n = (\delta^{-1} -
\delta^{-2})\mathrm{diag}\left(\Sigma_n^{(\delta)}\right) +
\delta^{-2}\Sigma_n^{(\delta)}.
\end{equation}
Our estimator is then solution of the following penalized empirical
risk minimization problem:
\begin{equation}\label{nuclearnormest}
\hat\Sigma^{\lambda} = \mathrm{argmin}_{S\in \mathcal
S_p}\left\|\tilde\Sigma_n - S\right\|_2^2 + \lambda \,\|S\|_1,
\end{equation}
where $\lambda>0$ is a regularization parameter to be tuned
properly. We note that this simple procedure can be computed
efficiently in high-dimension since $\hat\Sigma^{\lambda}$ is
solution of a convex minimization problem. The optimal choice of the
tuning parameter $\lambda$ is of the order of magnitude of the
stochastic deviation $\|\tilde\Sigma_n - \Sigma\|_\infty$.
Therefore, in order to order to establish sharp oracle inequalities
for (\ref{nuclearnormest}), we need first to study the deviations of
$\|\tilde\Sigma_n - \Sigma\|_\infty$. This analysis is more
difficult as compared to the study of $\|\Sigma_n - \Sigma\|_\infty$
since we need to derive the sharp scaling of $\|\tilde\Sigma_n -
\Sigma\|_\infty$ with $\delta$.

The rest of the paper is organized as follows. In Section
\ref{tools}, we recall some tools and definitions. In Section
\ref{oracle}, we establish oracle inequalities for the Frobenius and
spectral norms for our procedure (\ref{nuclearnormest}) and also
propose a data-driven choice of the regularization parameter. In
section \ref{lower}, we establish minimax lower bounds for data with
missing observations $\delta\in (0,1]$, thus showing that our
procedures are minimax optimal up to a logarithmic factor. Finally,
Section \ref{proof} contains all the proofs of the paper.

We emphasize that the results of this paper are non-asymptotic in
nature, hold true for any setting of $n,p$, are minimax optimal (up
to a logarithmic factor) and do not require the unknown covariance
matrix $\Sigma$ to be low-rank. We note also that to the best of our
knowledge, there exists in the literature no minimax lower bound
result for statistical problem with missing observations.

\section{Tools and definitions}\label{tools}

The $l_q$-norms of a vector
$x=\left(x^{(1)},\cdots,x^{(p)}\right)^{\top}\in \R^p$ is given by
$$
|x|_q = \left(\sum_{j=1}^p |x^{(j)}|^q\right)^{1/q},\; \text{for}\;1
\leq q < \infty, \quad \text{and}\quad |x|_\infty = \max_{1\leq j
\leq p}|x^{(j)}|.
$$

Denote by $\mathcal S_p$ the set of $p\times p$ symmetric
positive-semidefinite matrices. Any matrix $A\in \mathcal S_p$
admits the following spectral representation
$$
A = \sum_{j=1}^r \sigma_j(A) u_j(A)\otimes u_j(A)
$$
where $r=\mathrm{rank}(A)$ is the rank of $A$, $\sigma_1(A)\geq
\sigma_2(A)\geq \cdots \geq \sigma_r(A)> 0$ are the nonzero
eigenvalues of $A$ and $u_1(A),\ldots,u_r(A)\in \R^p$ are the
associated orthonormal eigenvectors (we also set $\sigma_{r+1}(A) =
\cdots=\sigma_p(A)=0$). The linear vector space $L$ is the linear
span of $\{u_1(A),\ldots,u_r(A)\}$ and is called support of $A$. We
will denote respectively by $P_L$ and $P_L^{\perp}$ the orthogonal projections onto $L$ and
$L^{\perp}$.

The Schatten $q$-norm of $A\in \mathcal S_p$ is defined by
$$
\|A\|_q = \left(\sum_{j=1}^p |\sigma_j(A)|^q\right)^{1/q},\;
\text{for}\;1 \leq q < \infty, \quad \text{and}\quad \|A\|_{\infty}
= \sigma_1(A),
$$
Note that the trace of any $S\in \mathcal S_p$ satisfies
$\mathrm{tr}(S) = \|S\|_1$.

Recall the {\it trace duality} property:
$$
\left| \mathrm{tr}(A^\top B) \right| \leq \|A\|_1 \|B\|_{\infty},
\quad \forall A,B\in \R^{p\times p}.
$$

We will also use the fact that the subdifferential of the convex
function $A\mapsto \|A\|_1$ is the following set of matrices :
\begin{equation}
\label{subdiff}
\partial \|A\|_1=\Bigl\{\sum_{j=1}^r u_j(A)
\otimes u_j(A) + P_{L}^{\perp}W P_{L}^{\perp}:\ \|W\|_\infty\leq 1
\Bigr\},
\end{equation}
(cf. \cite{watson}).

We recall now the definition and some basic properties of
sub-exponential random vectors.
\begin{defi}
The $\psi_\alpha$-norms of a real-valued random variable $V$ are
defined by
$$
\|V\|_{\psi_\alpha} = \inf\left\lbrace u>0: \mathbb E
\exp\left(|V|^\alpha/u^\alpha
 \right)\leq 2 \right\rbrace,\quad \alpha \geq 1.
$$
We say that a random variable $V$ with values in $\R$ is
sub-exponential if $\|V\|_{\psi_\alpha}<\infty$ for some $\alpha
\geq 1$. If $\alpha = 2$, we say that $V$ is sub-gaussian.
\end{defi}
We recall some well-known properties of sub-exponential random
variables:
\begin{enumerate}
\item For any real-valued random variable $V$ such that
$\|V\|_{\alpha}<\infty$ for some $\alpha \geq 1$, we have
\begin{equation*}
\sup_{m\geq 1} m^{-\frac{1}{\alpha}}\mathbb E \left(|V|^m
\right)^{1/m}\leq C \|V\|_{\psi_{\alpha}},
\end{equation*}
where $C>0$ can depend only on $\alpha$.
\item If a real-valued random variable $V$ is sub-gaussian, then
$V^2$ is sub-exponential. Indeed, we have
\begin{equation*}
\|V^2\|_{\psi_1}\leq 2 \|V\|_{\psi_2}^2.
\end{equation*}
\end{enumerate}

\begin{defi}
A random vector $X\in \R^p$ is sub-exponential if $\langle X,x\rangle$ are sub-exponential random variables for all
$x\in \R^p$. The $\psi_\alpha$-norms of a random vector $X$ are
defined by
$$
\|X\|_{\psi_\alpha} = \sup_{x\in\R^{p}:|x|_2=1}\|\langle
X,x\rangle\|_{\psi_\alpha},\quad \alpha\geq 1.
$$
\end{defi}

We recall the Bernstein inequality for sub-exponential real-valued
random variables (see for instance Corollary 5.17 in \cite{vershynin})
\begin{proposition}\label{Bernstein}
  Let $Y_1,\ldots,Y_n$ be independent centered sub-exponential
  random variables, and $K = \max_{i}\|Y_i\|_{\psi_1}$. Then for every $t\geq 0$, we have
  with probability at least $1-e^{-t}$
$$
\left| \frac{1}{n}\sum_{i=1}^n  Y_i  \right| \leq C K
\left(\sqrt{\frac{t}{n}}\vee \frac{t}{n}\right),
$$
where $C>0$ is an absolute constant.
\end{proposition}

The following proposition is the matrix version of Bernstein's
inequality for bounded random matrices \cite{Ahlswede} (see also Corollary 9.1 in \cite{tropp10}).

\begin{proposition}\label{prop:Bernstein_bounded}
Let $Z_1,\ldots,Z_n$ be symmetric independent random matrices in
$\R^{p\times p}$ that satisfy $\E(Z_i) = 0$ and $\|Z_i\|_\infty\leq
U$ almost surely for some constant $U$ and all $i=1,\dots,n$. Define
$$
\sigma_Z = \Big\|\frac1{n}\sum_{i=1}^n\E Z_i^2\Big\|_\infty.
$$
Then, for all $t>0,$ with probability at least $1-e^{-t}$ we have
$$
\left\| \frac{Z_1 + \cdots + Z_n}{n}  \right\|_\infty \leq
2\max\left\{ \sigma_Z\sqrt{\frac{t + \log (2p)}{n}}\,, \ \ U \frac{t
+ \log (2p)}{n} \right\}.
$$
\end{proposition}

\section{Oracle inequalities}\label{oracle}


We can now state the main result for the procedure
(\ref{nuclearnormest}).
\begin{theorem}\label{theomain1}
Let $X_1,\ldots,X_n$ be i.i.d. vectors in $\R^p$ with covariance
matrix $\Sigma$. For any $p\geq 2,n\geq 1$, we have on the event
$\lambda \geq 2\|\tilde \Sigma_n - \Sigma\|_\infty$
\begin{align}
\left\|  \hat\Sigma^\lambda - \Sigma  \right\|_2^2 &\leq \inf_{S\in
\mathcal S_p}\left\lbrace \left\| S - \Sigma \right\|_2^2 +
\min\left\lbrace  2\lambda \|S\|_1  ,
\frac{(1+\sqrt{2})^2}{8} \lambda^2\mathrm{rank}(S) \right\rbrace\right \rbrace,
\end{align}
and
\begin{eqnarray}
    \left\| \hat{\Sigma}^\lambda - \Sigma \right\|_{\infty} &\leq&
    \lambda.
  \end{eqnarray}
\end{theorem}
As we see in Theorem \ref{theomain1}, the regularization parameter
$\lambda$ should be chosen sufficiently large such that the
condition $\lambda \geq 2\|\tilde \Sigma_n - \Sigma\|_\infty$ holds
with probability close to $1$. The optimal choice of $\lambda$
depends on the unknown distribution of the observations. We consider
now the case of sub-gaussian random vector $X\in\R^p$.

\begin{assumption}[Sub-gaussian observations]\label{assumption1}
The random vector $X\in \R^p$ is sub-gaussian, that is $\|
X\|_{\psi_2}<\infty$. In addition, there exist a numerical constant $c_1>0$
such that
\begin{equation}\label{subexp1}
\E(\langle X, u \rangle)^2\geq c_1 \|\langle X, u
\rangle\|_{\psi_2}^2,\, \forall u\in \R^p.
\end{equation}
\end{assumption}
Note that Gaussian distributions satisfy Assumption
\ref{assumption1}. Under the above condition, we can study the
stochastic quantity $\|\tilde \Sigma_n - \Sigma\|_\infty$ and thus
properly tune the regularization parameter $\lambda$.

The intrinsic dimension of the matrix $\Sigma$ can be measured by
the effective rank
\begin{equation}\label{intrinsic}
\mathbf{r}(\Sigma):= \frac{\mathrm{tr}(\Sigma)}{\|\Sigma\|_\infty},
\end{equation}
see Section 5.4.3 in \cite{vershynin}. Note that we always have
$\mathbf{r}(\Sigma) \leq \mathrm{rank}(\Sigma)$. In addition, we can
possibly have $\mathbf{r}(\Sigma) \ll \mathrm{rank}(\Sigma)$ for
approximately low-rank matrices $\Sigma$, that is matrices $\Sigma$
with large rank but concentrated around a low-dimensional subspace. Consider for instance the covariance matrix $\Sigma$ with eigenvalues $\sigma_1 = 1$ and $\sigma_2 = \cdots = \sigma_p = 1/p$, then $\mathbf{r}(\Sigma) =\frac{2p-1}{p} \ll p = \mathrm{rank}(\Sigma)$ 

We have the following result, which requires no condition on the
covariance matrix $\Sigma$.
\begin{proposition}\label{lem1}
Let $X_1,\ldots,X_n\in\R^p$ be i.i.d. random vectors satisfying
Assumption \ref{assumption1}. Let $Y_1,\ldots,Y_n$ be defined in
(\ref{equationY}) with $\delta \in (0,1]$. Then, for any $t>0$, we
have with probability at least $1-e^{-t}$
\begin{equation}\label{prop1-bound-1}
\|\tilde\Sigma_n - \Sigma\|_\infty \leq
C\frac{\|\Sigma\|_\infty}{c_1} \max\left\{
\sqrt{\frac{\mathbf{r}(\Sigma)\left(t+ \log(2p)\right)}{\delta^2
n}},\frac{\mathbf{r}(\Sigma)\left(t+ \log(2p)\right)}{\delta^2
n}\left(c_1\delta + t+ \log n\right)\right\},
\end{equation}
and
\begin{equation}\label{prop1-bound-2}
 |\mathrm{tr}(\tilde\Sigma_n) - \mathrm{tr}(\Sigma)| \leq C
 \frac{\mathrm{tr}(\Sigma)}{c_1\delta}\max\left\{\sqrt{\frac{t}{n}},\frac{t}{n}\right\},
\end{equation}
where $C>0$ is an absolute constant.
\end{proposition}
\begin{enumerate}

\item The natural choice for $t$ is of the order of magnitude $\log
(2p)$. Then the conclusions of Proposition \ref{lem1} hold true with
probability at least $1-\frac{1}{2p}$. In addition, if the number of
measurements $n$ is sufficiently large
\begin{equation}\label{measurements}
n\geq c\frac{\mathbf{r}(\Sigma)}{\delta^2}\log^2((2p)\vee n),
\end{equation}
where $c>0$ is a sufficiently large numerical constant, then an
acceptable choice for the regularization parameter $\lambda$ is
\begin{equation}\label{lambdaoptimal}
\lambda = C \frac{\|\Sigma\|_\infty}{c_1}
\sqrt{\frac{\mathbf{r}(\Sigma)\log (2p)}{\delta^2 n}},
\end{equation}
where the absolute constant $C>0$ is sufficiently large.

\item As we claimed in the introduction, Proposition \ref{lem1} requires no condition on $\Sigma$ whatsoever. However, for the result to be of any practical interest, we need the bound in (\ref{prop1-bound-1}) to be small, which is the case if the condition (\ref{measurements}) is satisfied.
This condition is interesting since it shows that the number of
measurements sufficient to guarantee a precise enough estimation of
the spectrum of $\Sigma$ grows with the effective rank
$\mathbf{r}(\Sigma)$. In particular, when no observation is missing
($\delta =1$), if $\Sigma $ is approximately low-rank so that
$\mathbf{r}(\Sigma)\ll p$, then only
$n=\mathcal{O}\left(\mathbf{r}(\Sigma)\log^2(2p)\right)$
measurements are sufficient to estimate precisely the spectrum of
the $p\times p$ covariance matrix $\Sigma$.

  \item Note that if we assume that $\|Y\otimes Y\|_\infty = |Y|_2^2\leq
  U$ a.s. for some constant $U>0$, then we can eliminate the $(c_1\delta+t+\log n)$ factor in (\ref{prop1-bound-1}). Consequently, we can replace the condition (\ref{measurements}) on the number of measurements by the following less restrictive one
  $$
  n\geq c \frac{\mathbf{r}(\Sigma)}{\delta^2}\log(2p),
  $$
  for some absolute constant $c>0$ sufficiently large. When there is no missing observation ($\delta=1$), we obtain the standard condition on the number of measurements (see Remark 5.53 in \cite{vershynin}). When some observations are missing ($\delta<1$), we have the
additional quantity $\delta^2$ in the denominators of
(\ref{prop1-bound-1}) and (\ref{measurements}). The bound
(\ref{prop1-bound-1}) is degraded in the case $\delta<1$ since we
observe less entries per measurement. Consequently, as we can see it
in (\ref{measurements}), if we denote by $N(\epsilon)$ the number of
necessary measurements to estimate $\Sigma$ with a precision
$\epsilon$ when no observation is missing ($\delta =1$), then we
will need at least $\mathcal{O}\left(N(\epsilon)/\delta^2\right)$
measurements in order to estimate $\Sigma$ with the same precision
$\epsilon$ when some observations are missing ($\delta<1$). In
Theorem \ref{theomain2}, we prove in particular that the dependence
of the bound (\ref{prop1-bound-1}) on $\delta$ is sharp by
establishing a minimax lower bound.

\item In the full observations case ($\delta=1$) and for sub-gaussian distributions with low rank covariance matrix $\Sigma$,
 a simple modification of the $\epsilon$-net argument used in \cite{vershynin} to prove Theorem 5.39 yields an inequality similar to (\ref{prop1-bound-1}) 
with an upper bound of the order $\|\Sigma\|_{\infty}\sqrt{\frac{\mathrm{rank}(\Sigma)+t}{n}}$ without any logarithmic factor $\log 2p$. Note however that this bound is suboptimal when $\mathbf{r}(\Sigma)\log^2\left((2p)\vee n\right) \ll \mathrm{rank}(\Sigma)$ (cf the discussion below Assumption \ref{assumption1} on the intrinsic dimension of a matrix). In addition, in the missing observations framework $\delta<1$, the matrix $\Sigma^{(\delta)}$ can have full rank even if the matrix $\Sigma$ is low rank. Therefore the $\epsilon$-net argument will yield an upper bound of the order $\|\Sigma\|_\infty \sqrt{\frac{p+t}{\delta n}}$ which is much larger than the bound derived in (\ref{prop1-bound-1}).

\item Proposition \ref{lem1} and Equation (\ref{lambdaoptimal}) give some insight on the tuning of the
regularization parameter:
$$
\lambda = C
\frac{\sqrt{\mathrm{tr}(\Sigma)\|\Sigma\|_\infty}}{c_1\delta}
\sqrt{\frac{\log (2p)}{ n}},
$$
where $C>0$ is a sufficiently large absolute constant. We see that
this choice of $\lambda$ depends on $\mathrm{tr}(\Sigma)$ and
$\|\Sigma\|_\infty$ which are typically unknown. Therefore we
propose to use instead
\begin{equation}\label{lambdapractice}
\lambda = C
\frac{\sqrt{\mathrm{tr}(\tilde\Sigma_n)\|\tilde\Sigma_n\|_\infty}}{
\delta}\sqrt{\frac{\log 2p}{n}},
\end{equation}
where $C>0$ is a large enough constant. Note that the above choice
of $\lambda$ does not depend on the unknown quantities
$\|\Sigma\|_\infty$ or $\mathrm{tr}(\Sigma)$ and constitutes thus an
interesting choice in practice. We prove in the next lemma that
$2\|\tilde \Sigma_n - \Sigma\|_\infty\leq  \lambda$ with probability
at least $1- \frac{1}{2p}$.
\end{enumerate}

\begin{lemma}\label{lem-lambda-datadriven}
Let the assumptions of Proposition \ref{lem1} be satisfied. Assume
in addition that (\ref{measurements}) holds true. Take $\lambda$ as
in (\ref{lambdapractice}) with $C>0$ a large enough constant that
can depend only on $c_1$. Then, we have with probability at least
$1-\frac{1}{2p}$ that
$$2\|\tilde \Sigma_n -\Sigma\|_\infty \leq \lambda \leq C' \|\Sigma\|_\infty
\sqrt{\frac{\mathbf{r}(\Sigma)\log (2p)}{ \delta^2 n}},$$ where
$C'>0$ can depend only on $c_1$.
\end{lemma}

We obtain the following corollary of Theorem \ref{theomain1}.
\begin{corollary}\label{cor22}
Let Assumption \ref{assumption1} be satisfied. Assume that
(\ref{measurements}) is satisfied. Consider the estimator
(\ref{nuclearnormest}) with the regularization parameter $\lambda$
satisfying (\ref{lambdapractice}). Then we have, with probability at
least $1-\frac{1}{2p}$ that
\begin{equation}
\|\hat\Sigma^\lambda  - \Sigma\|_2^2 \leq  \inf_{S\in\mathcal
S_p}\left\lbrace \|\Sigma - S\|_2^2 +C_1 \|\Sigma\|_\infty^2
\frac{\mathbf{r}(\Sigma)\log 2p}{\delta^2
n}\mathrm{rank}(S)\right\rbrace,
\end{equation}
and
\begin{equation}
\|\hat\Sigma^\lambda  - \Sigma\|_\infty \leq C_2\|\Sigma\|_\infty
\sqrt{\frac{\mathbf{r}(\Sigma)\log 2p}{\delta^2 n}},
\end{equation}
where $C_1,C_2>0$ can depend only on $c_1$.
\end{corollary}

The proof of this corollary is immediate by combining Theorem
\ref{theomain1} with Proposition \ref{lem1} and Lemma
\ref{lem-lambda-datadriven}.

\section{Lower bounds}\label{lower}

For any integer $1\leq r \leq p$, define
$$\mathcal C_r =\left\lbrace
S \in \mathcal S_p\,:\, \mathbf{r}(S)\leq r \right\rbrace.$$ We also
introduce $\mathcal P_r$ the class of probability distributions on
$\R^p$ with covariance matrix $\Sigma\in \mathcal C_r$.

We now establish a minimax lower bound that guarantees the rates we
obtained in Corollary \ref{cor22} are optimal up to a logarithmic
factor on the probability distribution class $\mathcal P_r$. In particular, the dependence of our rates on $\delta$,
$\|\Sigma\|_\infty$ and $\mathbf{r}(\Sigma)$ is sharp.

\begin{theorem}\label{theomain2}
Fix $\delta \in (0,1]$. Let $n,r\geq 1$ be integers such that $n\geq
\delta^{-2} r^2$. Let $X_1,\ldots,X_n$ be i.i.d. random vectors in
$\R^p$ with covariance matrix $\Sigma\in \mathcal C_r$
. We
observe $n$ i.i.d. random vectors $Y_1,\ldots,Y_n\in\R^p$ such that
$$
Y_{i}^{j} = \delta_{ij}X_i^{(j)},\; 1\leq i \leq n,\; 1\leq j \leq
p,
$$
where $(\delta_{ij})_{1\leq i\leq n,\,1\leq j \leq p}$ is an i.i.d.
sequence of Bernoulli $B(\delta)$ random variables independent of
$X_1,\ldots,X_n$.

Then, there exist absolute constants $\beta\in(0,1)$ and $c>0$ such
that
\begin{equation}\label{eq:lower1}
\inf_{\hat{\Sigma}}
\sup_{\substack{\P_{\Sigma}\in\,{\cal P}_r
}}
\P_{\Sigma}\bigg(\|\hat{\Sigma}-\Sigma\|^2_{2}>
c\|\Sigma\|_\infty^2\frac{\mathbf{r}(\Sigma)}{\delta^2
n}\mathrm{rank}(\Sigma) \bigg)\ \geq\ \beta,
\end{equation}
and
\begin{equation}\label{eq:lower2}
\inf_{\hat{\Sigma}}
\sup_{\substack{\P_{\Sigma}\in\,{\cal P}_r
}}
\P_{\Sigma}\bigg(\|\hat{\Sigma}-\Sigma\|_{\infty}> c\|\Sigma\|_\infty
\sqrt{\frac{\mathbf{r}(\Sigma)}{\delta^2 n}} \bigg)\ \geq\ \beta,
\end{equation}
where $\inf_{\hat{\Sigma}}$ denotes the infimum over all possible
estimators $\hat{\Sigma}$ of $\Sigma$ based on $Y_1,\ldots,Y_n$.
\end{theorem}

\section{Proofs}\label{proof}

\subsection{Proof of Theorem \ref{theomain1}}
The proof of the first inequality adapts to covariance matrix
estimation the arguments used in the trace regression problem to
prove Theorems 1 and 11 in \cite{KLT}.
\begin{proof}
By definition of $\hat \Sigma^\lambda$, we have for any $S\in
\mathcal S_p$
\begin{align*}
\|\hat\Sigma^\lambda - \Sigma\|_2^2 \leq \|S - \Sigma\|_2^2 +
\lambda \|S\|_1 + 2\langle \Sigma - \tilde\Sigma_n,S-\hat\Sigma^\lambda
\rangle -\lambda \|\hat \Sigma^\lambda\|_1.
\end{align*}
If $\lambda \geq 2 \|\tilde\Sigma_n - \Sigma\|_\infty$, we deduce from the
previous display that
\begin{align*}
\|\hat\Sigma^\lambda - \Sigma\|_2^2 \leq \|S - \Sigma\|_2^2 +2
\lambda \|S\|_1,\quad \forall S\in \mathcal S_p.
\end{align*}
Next, a necessary and sufficient condition of minimum for problem
(\ref{nuclearnormest}) implies that there exists $\hat V \in
\partial \|
\hat \Sigma^\lambda\|_1$ such that for all $S\in \mathcal S_p$
\begin{equation}\label{KKT}
-2\langle \tilde\Sigma_n - \hat\Sigma^\lambda , \hat\Sigma^\lambda - S
\rangle + \lambda \langle \hat V, \hat\Sigma^\lambda - S  \rangle
\leq 0.
\end{equation}
For any $S\in \mathcal S_p$ of rank $r$ with spectral representation
$S = \sum_{j=1}^r \sigma_ju_j\otimes u_j$ and support $L$, It
follows from (\ref{KKT}) that
\begin{align}\label{Th1-interm1}
2\langle \hat\Sigma^\lambda -\Sigma, \hat\Sigma^\lambda - S \rangle
+ \lambda \langle \hat V-V, \hat\Sigma^\lambda - S  \rangle \leq
-\lambda \langle V,\hat\Sigma^\lambda -  S \rangle + 2\langle
\tilde\Sigma_n - \Sigma, \hat\Sigma^\lambda -  S \rangle,
\end{align}
for an arbitrary $V\in \partial \|S\|_1$. Note that $\langle \hat
V-V, \hat\Sigma^\lambda - S \rangle \geq 0$ by monotonicity of
subdifferentials of convex functions and that the following
representation holds
$$
V = \sum_{j=1}^r u_j \otimes u_j + P_{L}^{\perp}WP_{L}^{\perp},
$$
where $W$ is an arbitrary matrix with $\|W\|_\infty\leq 1$. In particular,
there exists $W$ with $\|W\|_\infty\leq 1$ such that
$$
\langle P_{L}^{\perp}WP_{L}^{\perp} , \hat\Sigma^\lambda - S \rangle
= \| P_{L}^{\perp}\hat\Sigma^\lambda P_{L}^{\perp} \|_1.
$$
For this choice of $W$, we get from (\ref{Th1-interm1}) that
\begin{align}\label{Th1-interm2}
\|\hat\Sigma^\lambda - \Sigma\|_2^2  +  &\|\hat\Sigma^\lambda -
S\|_2^2 +\lambda \| P_{L}^{\perp}\hat\Sigma^\lambda P_{L}^{\perp}
\|_1 \leq  \|S - \Sigma \|_2^2\notag \\&\hspace{2cm}+\lambda
\|P_{L}(\hat\Sigma^\lambda - S)P_{L}\|_1+ 2\langle \tilde\Sigma_n -
\Sigma, \hat\Sigma^\lambda - S \rangle,
\end{align}
where we have used the following facts
\begin{align*}
2\langle \hat\Sigma^\lambda -\Sigma, \hat\Sigma^\lambda - S
\rangle=\|\hat\Sigma^\lambda - \Sigma\|_2^2  +  \|\hat\Sigma^\lambda
- S\|_2^2-\|S - \Sigma \|_2^2,
\end{align*}
and
$$
\left\|\sum_{j=1}^r u_j\otimes u_j\right\|_\infty = 1,\quad
\left\langle \sum_{j=1}^r u_j\otimes u_j,\hat\Sigma^\lambda -  S
\right\rangle = \left\langle \sum_{j=1}^r u_j\otimes
u_j,P_{L}(\hat\Sigma^\lambda - S )P_L\right\rangle.
$$
For any $A\in \R^{p\times p}$ define $\mathcal P_L(A) = A -
P_{L}^{\perp} A P_{L}^{\perp}$. Set $\Delta_1 = \tilde\Sigma_n - \Sigma$. We have
$$
\langle \Delta_1, \hat\Sigma^\lambda - S \rangle = \langle \Delta_1,
\mathcal P_{L}(\hat\Sigma^\lambda - S) \rangle + \langle \Delta_1,
P_{L}^{\perp}(\hat\Sigma^\lambda - S)P_{L}^{\perp} \rangle.
$$
Using Cauchy-Schwarz's inequality and trace duality, we get
\begin{align*}
|\langle \Delta_1, \mathcal P_{L}(\hat\Sigma^\lambda - S) \rangle|
&\leq \sqrt{2\mathrm{rank}(S)}\|\Delta_1\|_\infty\|\hat\Sigma^\lambda - S\|_2,\\
\|P_{L}(\hat\Sigma^\lambda - S)P_{L}\|_1 &\leq
\sqrt{\mathrm{rank}(S)}\|\hat\Sigma^\lambda - S\|_2,\\
|\langle \Delta_1, P_{L}^{\perp}(\hat\Sigma^\lambda -
S)P_{L}^{\perp} \rangle| &\leq
\|\Delta_1\|_\infty\|P_{L}^{\perp}\hat\Sigma^\lambda P_{L}^{\perp}\|_1.
\end{align*}
The above display combined with (\ref{Th1-interm2}) give
\begin{align*}
\|\hat\Sigma^\lambda - \Sigma\|_2^2  +  &\|\hat\Sigma^\lambda -
S\|_2^2 +(\lambda - 2\|\Delta_1\|_\infty) \|
P_{L}^{\perp}\hat\Sigma^\lambda P_{L}^{\perp} \|_1 \leq  \|S -
\Sigma \|_2^2\notag \\&\hspace{4cm}+ (\sqrt{2}\|\Delta_1\|_\infty+\lambda)
\sqrt{r}\|\hat\Sigma^\lambda - S\|_2
\end{align*}
A decoupling argument gives
\begin{align*}
\|\hat\Sigma^\lambda - \Sigma\|_2^2  +  &\|\hat\Sigma^\lambda -
S\|_2^2 +(\lambda - 2\|\Delta_1\|_\infty) \|
P_{L}^{\perp}\hat\Sigma^\lambda P_{L}^{\perp} \|_1 \leq  \|S -
\Sigma \|_2^2\notag \\&\hspace{4cm}+
\left(\frac{1}{\sqrt{2}}\|\Delta_1\|_\infty+\frac{\lambda}{2}\right)^2
r + \|\hat\Sigma^\lambda - S\|_2^2.
\end{align*}
Finally, we get on the event $\lambda \geq 2 \|\Delta_1\|_\infty$
that
\begin{align*}
\|\hat\Sigma^\lambda - \Sigma\|_2^2  \leq  \|S - \Sigma \|_2^2
+\frac{(1+\sqrt{2})^2}{8}\lambda^2 \mathrm{rank}(S),\quad \forall
S\in \mathcal S_p.
\end{align*}

We now prove the spectral norm bound. Note first that the solution
of (\ref{nuclearnormest}) is given by
\begin{equation}\label{solution}
\hat\Sigma^{\lambda} = \sum_j \left(\sigma_j(\tilde\Sigma_n)-
\frac{\lambda}{2}\right)_+ u_j(\tilde\Sigma_n)\otimes
u_j(\tilde\Sigma_n),
\end{equation}
where $x_+ =\max\{0,x\}$ and $\tilde\Sigma_n$ admits the spectral
representation $$\tilde\Sigma_n = \sum_j
\sigma_j(\tilde\Sigma_n)u_j(\tilde\Sigma_n)\otimes u_j(\tilde\Sigma_n),$$
with positive eigenvalues $\sigma_j(\tilde\Sigma_n)\geq 0$ and orthonormal eigenvectors
$u_j(\tilde\Sigma_n)$. Indeed, the solution of (\ref{nuclearnormest}) is
unique since the functional $S\rightarrow F(S) = \|\tilde\Sigma_n - S\|_2^2 +
\lambda \|S\|_1$ is strictly convex. 
A sufficient condition of minimum is $\textbf{0} \in
\partial F (\hat \Sigma^\lambda)=-2( \tilde\Sigma_n
- \hat\Sigma^{\lambda})+ \lambda \hat V$ with $\hat V \in \partial
\|\hat \Sigma^\lambda\|_1$. We consider the following choice of
$\hat V =
\sum_{j:\sigma_j(\tilde\Sigma_n)\geq\lambda/2}u_j(\tilde\Sigma_n)\otimes
u_j(\tilde\Sigma_n) + W\in
\partial \|\hat\Sigma^\lambda\|_1$ with
$$
W = \sum_{j:\sigma_j(\tilde\Sigma_n)<\lambda/2}
\frac{2\sigma_j(\tilde\Sigma_n)}{\lambda} u_j(\tilde\Sigma_n)\otimes
u_j(\tilde\Sigma_n).
$$
It is easy to check that $\partial F(\hat\Sigma^\lambda)=-2(
\tilde\Sigma_n - \hat\Sigma^{\lambda})+ \lambda \hat V =\textbf{0}$.

Next, we have on the event $\lambda \geq 2 \|\Delta_1\|_\infty$
\begin{align*}
\|\hat\Sigma^\lambda - \Sigma\|_\infty &\leq \|\hat\Sigma^\lambda -
\tilde\Sigma_n\|_\infty + \| \Delta_1 \|_\infty\leq  \lambda.
\end{align*}
\end{proof}

\subsection{Proof of Proposition \ref{lem1}}

The delicate part of this proof is to obtain the sharp dependence on
$\delta$. As a consequence, the proof is significantly more
technical as compared to the case of full observations $\delta=1$.
To simplify the understanding of this proof, we decomposed it into
three lemmas that we prove below.

\begin{proof}
Define
$$
A_n^{(\delta)} =
\Sigma_n^{(\delta)}-\mathrm{diag}\left(\Sigma_n^{(\delta)}\right),\quad
A^{(\delta)} =
\Sigma^{(\delta)}-\mathrm{diag}\left(\Sigma^{(\delta)}\right).
$$

We have
\begin{align*}
  \left\|\tilde \Sigma_n - \Sigma\right\|_\infty &\leq
  \delta^{-1}\left\|\mathrm{diag}\left(\Sigma_n^{(\delta)} -
  \Sigma^{(\delta)}\right)\right\|_\infty + \delta^{-2}\left\|A_n^{(\delta)} -
  A^{(\delta)}\right\|_\infty.
\end{align*}
Now combining a simple union bound argument with Lemmas \ref{lem3},
\ref{lem4} and \ref{lem5}, we get with probability at least
$1-4e^{-t}$ that
\begin{align*}
|\mathrm{tr}(\Sigma_n^{(\delta)}) - \delta\mathrm{tr}(\Sigma)| \leq
Cc_1^{-1}\mathrm{tr}(\Sigma)\max\left\{\sqrt{\frac{t}{n}},\frac{t}{n}\right\},
\end{align*}
and
\begin{align*}
\|\tilde \Sigma_n - \Sigma\|_\infty \leq Cc_1^{-1}&\left[
\max_{1\leq j\leq p}(\Sigma_{jj})\max\left( \sqrt{\frac{t+\log
p}{\delta^{2}n}}, \frac{t+\log p}{\delta n}\right)\right.\\
&\hspace{-2cm}+\left. \max\left(
\sqrt{\mathrm{tr}(\Sigma)\|\Sigma\|_\infty \frac{t+\log
(2p)}{\delta^{2}n}}, \mathrm{tr}(\Sigma)\left(c_1\delta+ t+\log
n\right)\frac{t+\log (2p)}{\delta^{2}n} \right) \right],
\end{align*}
where $C>0$ is an absolute constant.

Noting finally that $\max_{1\leq j \leq p} (\Sigma_{jj})\leq
\sqrt{\mathrm{tr}(\Sigma)\|\Sigma\|_\infty} \wedge
\mathrm{tr}(\Sigma)$ and $\mathrm{tr}(\tilde \Sigma_n) = \delta^{-1}
\mathrm{tr}(\Sigma_n^{(\delta)})$, we can conclude, up to a
rescaling of the absolute constant $C>0$, that (\ref{prop1-bound-1})
and (\ref{prop1-bound-2}) hold true simultaneously with probability
at least $1-e^{-t}$.
\end{proof}

\begin{lemma}\label{lem3}
Under the assumptions of Proposition \ref{lem1}, we have with
probability at least $1-e^{-t}$ that
\begin{equation}
\left\|\mathrm{diag}\left(\Sigma_n^{(\delta)} -
  \Sigma^{(\delta)}\right)\right\|_\infty \leq C c_1^{-1}\max_{1\leq j \leq
  p}\left( \Sigma_{jj} \right) \max\left( \sqrt{\frac{t+\log p}{n}},
  \frac{t+\log p}{n}\right).
\end{equation}
where $C>0$ is an absolute constant.
\end{lemma}

\begin{proof}
We have
\begin{align*}
\left\| \mathrm{diag}\left(\Sigma_n^{(\delta)} -
\Sigma^{(\delta)}\right) \right\|_\infty &= \max_{1\leq j \leq p}
\left| \frac{1}{n}\sum_{i=1}^{n}\delta_{i,j}^2 \left( X_i^{(j)}
\right)^2 - \delta \Sigma_{jj} \right|.
\end{align*}
Next, since the random variables $\delta_{i,j}$ and $X_{i}^{(j)}$
are sub-gaussian for any $i,j$, we have
\begin{align*}
\left\| \left(\delta_{i,j}X_i^{(j)}\right)^2 \right\|_{\psi_1} \leq
2\left\|\delta_{i,j}X_i^{(j)}\right\|_{\psi_2}^2 \leq 2
\left\|X_i^{(j)}\right\|_{\psi_2}^2\leq 2c_1^{-1} \Sigma_{jj},
\end{align*}
where we have used Assumption 1 in the last inequality. We can apply
Bernstein's inequality (see Proposition \ref{Bernstein} in the
appendix below) to get for any $1\leq j \leq p$ with probability at
least $1-e^{-t'}$ that
\begin{align*}
\left| \frac{1}{n}\sum_{i=1}^{n}\delta_{i,j}^2 \left( X_i^{(j)}
\right)^2 - \delta \Sigma_{jj} \right| \leq Cc_1^{-1} \Sigma_{jj}
\max\left\{ \sqrt{\frac{t'}{n}}, \frac{t'}{n} \right\},
\end{align*}
where $C>0$ is an absolute constant. Next, taking $t'=t+\log p$
combined with a union bound argument we get the result.
\end{proof}

\begin{lemma}\label{lem4}
Under the assumptions of Proposition \ref{lem1}, we have with
probability at least $1-2e^{-t}$ that
\begin{align}
\|A_n^{(\delta)} - A^{(\delta)}\|_\infty &\leq C
\frac{\|\Sigma\|_\infty}{c_1} \max\left\{ \delta
\sqrt{\frac{\mathbf{r}(\Sigma)\left(t + \log (2p)\right)}{n}}\,,\right.\notag\\
\ &\hspace{2cm}\left. \frac{\mathbf{r}(\Sigma)\left(t + \log
(2p)\right)}{n}\left( c_1\delta + t+\log n\right) \right\},
\end{align}
where $C>0$ is a large enough absolute constant.
\end{lemma}

\begin{proof}
We have
$$
A_n^{(\delta)} - A^{(\delta)} = \frac{1}{n}\sum_{i=1}^n Z_i -
\mathbb E \left(Z_i \right),
$$
where
$$
Z_i = Y_i\otimes Y_i - \mathrm{diag}\left( Y_i\otimes Y_i \right)
,\; 1\leq i \leq n.
$$
Define $Y=(\delta_1X^{(1)},\ldots,\delta_pX^{(p)})^{\top}$ where
$\delta_1,\ldots,\delta_p$ are i.i.d. Bernoulli random variables
with parameter $\delta$ independent from $X$ and $Z=Y\otimes Y-
\mathrm{diag}\left( Y\otimes Y \right).$
We want to apply the noncommutative Bernstein inequality for
matrices. To this end, we need to study the quantities $\|\mathbb E
(Z-\mathbb E Z)^2\|_\infty$ and $\|Z-\mathbb E Z\|_\infty$.

We note first that $\|\mathbb E (Z-\mathbb E Z)^2\|_\infty \leq
\|\mathbb E Z^2\|_\infty$. Next, we set $V= Z+ \delta
\mathrm{diag}(X\otimes X)$ and $W=\delta \mathrm{diag}(X\otimes X)$.
Some easy algebra yields that
\begin{align}\label{interm-lem4-1}
  \|  \mathbb E Z^2 \|_\infty \leq \left( \sqrt{\| \mathbb E V^2 \|_\infty} + \sqrt{\|\mathbb E W^2\|_\infty}
  \right)^2.
\end{align}

We now treat $\| \mathbb E V^2 \|_\infty$ and $\|\mathbb E
W^2\|_\infty$ separately. Denote by $\mathbb E_{\delta}$ and
$\mathbb E_X$ the expectations w.r.t. $(\delta_1,\cdots,\delta_p)$
and $X$ respectively. We have $\mathbb E V^2 = \mathbb E_X \mathbb
E_\delta V^2$. Next, we have
\begin{eqnarray*}
\left(\mathbb E_\delta V^2\right)_{k,l} &=&
\begin{cases}
\delta^2 (X^{(k)})^2 |X|_2^2 &\text{if $k = l$},\\
\delta^3 X^{(k)}X^{(l)}|X|_2^2 &\text{otherwise}.
\end{cases}
\end{eqnarray*}
Consequently, we get for any $u=(u_1,\cdots,u_p)^{\top}\in \R^p$
with $|u|_2^2 = 1$ that
\begin{align}\label{interm-lem4-2}
\mathbb E u^{\top}V^2u &= \delta^2\left( \delta \mathbb E_X\left[
|X|_2^2 (X^\top u)^2 \right] + (1-\delta) \mathbb E_X \left[
\sum_{j=1}^p |X|_2^2(X^{(j)})^2 u_j^2\right]\right)\notag\\
&\leq \delta^2 \sqrt{\mathbb E_X |X|_2^4} \left(  \delta
\sqrt{\mathbb E_X (X^{\top}u)^4} + (1-\delta) \sum_{j=1}^p
\sqrt{\mathbb E_X (X^{(j)})^4}u_j^2 \right),
\end{align}
where we have applied Cauchy-Schwarz's inequality.


We have again by Cauchy-Schwarz's inequality and Assumption
\ref{assumption1} that
\begin{eqnarray*}
\mathbb E_X |X|_2^4
   &=& \sum_{j=1}^p\mathbb E (X^{(j)})^4 + \sum_{j, k = 1:j\neq k}^p
  \mathbb E (X^{(j)})^2(X^{(k)})^2\\
  &\leq& \sum_{j=1}^p\mathbb E (X^{(j)})^4 + \sum_{j, k = 1\,:\,j\neq k}^p
  \left[\mathbb E (X^{(j)})^4\right]^{1/2}\left[\mathbb E (X^{(k)})^4\right]^{1/2}\\
  &\leq& \left(\sum_{j=1}^p \sqrt{\mathbb E (X^{(j)})^4}\right)^2\\
  &\leq& C\left(\sum_{j=1}^p \|X^{(j)}\|_{\psi_2}^2\right)^2\\
  &\leq& Cc_1^{-2}\left(\mathrm{tr}(\Sigma)\right)^2,
\end{eqnarray*}
for some absolute constant $C>0$.
We have also, in view of (\ref{subexp1}), with the same absolute constant $C$ as above
$$
\sqrt{\mathbb E_X \langle X,u \rangle ^4} \leq C \|\langle X,u
\rangle\|_{\psi_2}^2 \leq C c_1^{-1}\|\Sigma\|_\infty,\quad \forall u \in
\R^p\;:\; |u|_2 = 1,
$$
and
$$
\sqrt{\mathbb E_X (X^{(j)})^4} \leq C \| X^{(j)}\|_{\psi_2}^2 \leq
C c_1^{-1}\Sigma_{jj},\quad 1\leq j \leq p.
$$

Combining the three above displays with (\ref{interm-lem4-2}), we
get
\begin{align}\label{interm-lem4-3}
\|\mathbb E V^2\|_\infty &\leq C c_1^{-1}\delta^2 \mathrm{tr}(\Sigma) \left[
\delta \|\Sigma\|_\infty + (1-\delta) \max_{1\leq j \leq p}
(\Sigma_{jj}) \right]\notag\\
&\leq C c_1^{-2}\delta^2 \mathrm{tr}(\Sigma)\|\Sigma\|_\infty,
\end{align}
and
\begin{align*}
\| \mathbb E W^2 \|_\infty &= \delta^2 \max_{1\leq j \leq p} \mathbb
E_X \left(  X^{(j)} \right)^4 \leq Cc_1^{-2}\delta^2 \max_{1\leq j \leq
p}(\Sigma_{jj}^2) \leq C c_1^{-2}\delta^2
\mathrm{tr}(\Sigma)\|\Sigma\|_\infty.
\end{align*}

Combining the two above displays with (\ref{interm-lem4-1}), we get
$$
\| \mathbb E Z^2 \|_\infty \leq  C c_1^{-2}\delta^2
\mathrm{tr}(\Sigma)\|\Sigma\|_\infty,
$$
where $C>0$ is an absolute constant.

Next, we treat $\|  Z - \mathbb E Z \|_\infty$. We have
\begin{align*}
\|Z-\mathbb E Z\|_\infty &\leq \|Z\|_\infty + \|\mathbb E
Z\|_\infty\leq |Y|_2^2 + \delta^2 \|\Sigma\|_\infty,
\end{align*}
where we have used that $\|Z\|_\infty \leq \|Y\otimes Y\|_\infty =
|Y|_2^2$ and
$$
\|\mathbb E Z\|_\infty = \|\Sigma^{(\delta)} -
\mathrm{diag}(\Sigma^{(\delta)})\|_\infty =\delta^2 \|\Sigma -
\mathrm{diag}(\Sigma)\|_\infty \leq \delta^2 \|\Sigma\|_\infty.
$$

In view of Assumption \ref{assumption1}, we have
$$
\left\||Y|_2^2\right\|_{\psi_{1}} \leq \sum_{j=1}^p
\left\|\delta_j(X^{(j)})^2\right\|_{\psi_1} \leq 2 \sum_{j=1}^p
\left\|X^{(j)}\right\|_{\psi_2}^2 \leq 2c_1^{-1}
\mathrm{tr}(\Sigma).
$$


Then, combining Proposition \ref{Bernstein} with a union bound
argument gives for any $t>0$
\begin{equation*}
\mathbb P \left( \max_{1\leq i \leq n}|Y_i|_2^2 \leq
\mathrm{tr}(\Sigma)\left( \delta + Cc_1^{-1}(t+\log n)\right)\right)
\geq 1-e^{-t},
\end{equation*}
where $C>0$ is an absolute constant.

Define
$$
U = \mathrm{tr}(\Sigma)\left( C^{-1}c_1\delta + t+\log n\right),
$$
and
$$t_1 = C' c_1^{-1}\max\left\{ \delta
\sqrt{\mathrm{tr}(\Sigma)\|\Sigma\|_\infty}\sqrt{\frac{t + \log
(2p)}{n}}\,, \ \ \mathrm{tr}(\Sigma)\left( c_1\delta + t+\log
n\right) \frac{t + \log (2p)}{n} \right\},
$$
where $C'>0$ is a large enough absolute constant.

We have
\begin{align*}
\mathbb P \left(  \| A_n^{(\delta)} -  A^{(\delta)}\|_\infty \geq
t_1 \right) &\leq \mathbb P \left(  \left\{\| A_n^{(\delta)} -
A^{(\delta)}\|_\infty \geq t_1 \right\} \cap \bigcap_{i=1}^n
\left\{|Y_i|_2^2 \leq U\right\} \right)\\
&\hspace{5cm}+ \mathbb P \left( \bigcup_{i=1}^n \left\{|Y_i|_2^2 >
U\right\} \right)\\
&\leq  \mathbb P \left( \| A_n^{(\delta)} - A^{(\delta)}\|_\infty
\geq t_1 \mid  \bigcap_{i=1}^n \left\{|Y_i|_2^2 \leq U\right\}
\right)+e^{-t}\\
&\leq 2 e^{-t},
\end{align*}
where we have used Proposition \ref{prop:Bernstein_bounded} to get
that
\begin{align*}
\mathbb P \left(  \left\{\| A_n^{(\delta)} - A^{(\delta)}\|_\infty >
t_1 \right\} \mid \bigcap_{i=1}^n \left\{|Y_i|_2^2 \leq U\right\}
\right) \leq e^{-t}.
\end{align*}

\end{proof}

\begin{lemma}\label{lem5}
Under the assumptions of Proposition \ref{lem1}, we have with
probability at least $1-e^{-t}$ that
\begin{equation}
|\mathrm{tr}(\Sigma_n^{(\delta)}) - \delta\mathrm{tr}(\Sigma)|\leq C
c_1^{-1} \mathrm{tr}(\Sigma) \max\left( \sqrt{\frac{t}{n}},
\frac{t}{n}\right),
\end{equation}
where $C>0$ is an absolute constant.
\end{lemma}

\begin{proof}
In view of Assumption \ref{assumption1}, we have for any $1\leq j
\leq p$ that $\|(Y^{(j)})^2\|_{\psi_1} \leq
\|(X^{(j)})^2\|_{\psi_{1}} \leq 2 \|X^{(j)}\|_{\psi_2}^2 \leq
2c_1^{-1}\Sigma_{jj}$ and
$$
\left\||Y|_2^2\right\|_{\psi_{1}} \leq \sum_{j=1}^p
\left\|(Y^{(j)})^2\right\|_{\psi_1} \leq
\left\|(X^{(j)})^2\right\|_{\psi_1} \leq 2 \sum_{j=1}^p
\left\|X^{(j)}\right\|_{\psi_2}^2 \leq 2c_1^{-1}
\mathrm{tr}(\Sigma).
$$

Next, we have
\begin{eqnarray*}
  \mathrm{tr}(\Sigma_n^{(\delta)}) -\delta\mathrm{tr}(\Sigma) &=& \mathrm{tr}(\Sigma_n^{(\delta)} - \Sigma^{(\delta)})\\
  &=& \mathrm{tr}\left(\frac{1}{n}\sum_{i=1}^n X_i\otimes X_i -\mathbb E (X\otimes
  X)\right)\\
  &=& \frac{1}{n}\sum_{i=1}^n \mathrm{tr}(Y_i\otimes Y_i) -\mathbb E \left(\mathrm{tr}(Y\otimes Y)\right)\\
  &=& \frac{1}{n}\sum_{i=1}^n |Y_i|_{2}^2 -\mathbb E  |Y|_2^2.
\end{eqnarray*}
Next, we have
$$
\| |Y_i|_{2}^2 -\mathbb E  |Y|_2^2 \|_{\psi_1}\leq \| |Y_i|_{2}^2
\|_{\psi_1} \leq 2c_1^{-1} \mathrm{tr}(\Sigma).
$$
Then we can apply Proposition \ref{Bernstein} to get the result.
\end{proof}

\subsection{Proof of Lemma \ref{lem-lambda-datadriven}}
In view of Proposition \ref{lem1}, we have on an event $\mathcal A$
of probability at least $1-\frac{1}{2p}$ that
\begin{equation}\label{lem-lambda-datadriven-interm1}
\|\tilde \Sigma_n - \Sigma\|_\infty\leq
C\frac{\|\Sigma\|_\infty}{c_1} \sqrt{\frac{\mathbf{r}(\Sigma)\log
2p}{\delta^2 n}}.
\end{equation}
We assume further that (\ref{measurements}) is satisfied with a
sufficiently large constant $c$ so that we have, in view of
(\ref{prop1-bound-1}) and (\ref{prop1-bound-2}), on the same event
$\mathcal A$ that
$$
\|\tilde \Sigma_n - \Sigma\|_\infty\leq \frac{\|\Sigma\|_\infty}{2}
$$
and
$$
|\mathrm{tr}(\tilde \Sigma_n) - \mathrm{tr}(\Sigma) | \leq
\frac{\mathrm{tr}(\Sigma)}{2}.
$$
We immediately get on the event $\mathcal A$ that
$$
\frac{1}{2}\|\Sigma\|_\infty \leq  \|\tilde \Sigma_n\|_\infty\leq
\frac{3}{2}\|\Sigma\|_\infty,
$$
and
$$
\frac{1}{2}\mathrm{tr}(\Sigma) \leq \mathrm{tr}(\tilde\Sigma_n) \leq
\frac{3}{2}\mathrm{tr}(\Sigma).
$$
Combining these simple facts with
(\ref{lem-lambda-datadriven-interm1}), we get the result.

\subsection{Proof of Theorem \ref{theomain2}}

This proof uses standard tools of the minimax theory (cf. for
instance \cite{tsy_09}). However, as for Proposition \ref{lem1}, the
proof with missing observations ($\delta<1$) is significantly more
technical as compared to case of full observations ($\delta=1$). In
particular, the control of the Kullback-Leibler divergence requires
a precise description of the conditional distributions of the random
variables $Y_1,\ldots, Y_n$ given the masked variables
$\delta_1,\ldots,\delta_n$. To our knowledge, there exists no
minimax lower bound result for statistical problem with missing
observations in the literature.

\begin{proof}

Set $\gamma = a/\sqrt{\delta^2 n}$ where $a>0$ is a sufficiently
small absolute constant.

We consider first the case $r\geq 2$. Define
\begin{equation*}
  \mathcal{N} = \left\{ E_{k,l} + E_{l,k},1 \leq k\leq r-1,\, k+1\leq
  l \leq r \right\}. 
\end{equation*}
Set $B_{k,l} = E_{k,l} + E_{l,k}$ for any $1 \leq k\leq r-1,\,
k+1\leq  l \leq r$. Consider the associated set of symmetric
matrices
\begin{equation*}
  \mathcal{B}(\mathcal N) = \left\{  \Sigma_\epsilon =
  \left(\begin{array}{c|c}
I_r+ \gamma\sum_{k=1}^{r-1}\sum_{l=k+1}^{r}\epsilon_{k,l} B_{k,l} & O \\
\hline O & O
  \end{array}\right),\; \epsilon= (\epsilon_{k,l})_{k,l} \in \{0,1\}^{\frac{r(r-1)}{2}}\right\},
\end{equation*}
Note that any matrix $\Sigma_\epsilon \in \mathcal{B}(\mathcal N)$
is positive-semidefinite if $0<a<1$ since we have by assumption
\begin{align}\label{spectraltest}
\left\|\gamma\sum_{k=1}^{r-1}\sum_{l=k+1}^{r}\epsilon_{k,l} B_{k,l}\right\|_\infty\leq \gamma r = a\sqrt{\frac{r^2}{\delta^2 n}}
\leq a.
\end{align}

By construction, any element of $\mathcal B(\mathcal N)$ as
well as the difference of any two elements of $\mathcal B(\mathcal
N)$ is of rank exactly $r$. Consequently, $\mathcal B(\mathcal
N)\subset \mathcal C_r$ since $\mathbf{r}(\Sigma_\epsilon)\leq
\mathrm{rank}(\Sigma_\epsilon) \leq r$ for any $\Sigma_\epsilon \in
\mathcal B(\mathcal N)$. Note also that for any $\Sigma_\epsilon\in
\mathcal B(\mathcal N)$, we have $\mathrm{tr}(\Sigma_\epsilon) = r$
and $ 0<1-a \leq \|\Sigma_\epsilon\|_\infty \leq 1+a$ provided that
$0<a<1$ and consequently $r/(1+a)\leq \mathbf{r}(\Sigma_\epsilon)
\leq r/(1-a)$. Indeed, we have
\begin{align*}
\|\Sigma_\epsilon \|_{\infty} &\leq 1+ \gamma
\left\|\sum_{k=1}^{r-1}\sum_{l=k+1}^{r}\epsilon_{k,l}
B_{k,l}\right\|_\infty\leq 1+ \gamma r\leq 1+a,
\end{align*}
in view of the condition $n\geq \delta^{-2}r^2$. A similar reasoning
gives the lower bound.

Denote by $A_0$ the $p\times p$ block matrix with first block equal
to $I_r$. Varshamov-Gilbert's bound (cf. Lemma 2.9 in \cite{tsy_09})
guarantees the existence of a subset $\mathcal
A^0\subset\mathcal{B}(\mathcal{N})$ with cardinality
$\mathrm{Card}(\mathcal A^0) \geq 2^{r(r-1)/16}+1$ containing $A_0$
and such that, for any two distinct elements $\Sigma_\epsilon$ and
$\Sigma_{\epsilon'}$ of $\mathcal A^0$, we have
\begin{align}\label{lower_2}
\Arrowvert \Sigma_\epsilon-\Sigma_{\epsilon'}\Arrowvert_{2}^2 &\geq
\gamma^2\frac{r(r-1)}{8}\geq
\gamma^2\frac{r^2}{16}=\frac{a^2}{16}\frac{r^2}{\delta^2 n}\notag\\
 &\geq
\frac{(1-a)a^2}{16(1+a)^2}\|\Sigma_{\tilde\epsilon}\|_\infty^2\frac{\mathbf{r}(\Sigma_{\tilde\epsilon})}{\delta^2
n}\mathrm{rank}(\Sigma_{\tilde\epsilon}),\quad\forall
\Sigma_{\tilde\epsilon}\in \mathcal A^0.
\end{align}

Let $X_1,\ldots,X_n\in\R^p$ be i.i.d.
$N\left(0,\Sigma_\epsilon\right)$ with $\Sigma_\epsilon\in \mathcal
A^{0}$. For the sake of brevity, we set $\Sigma =\Sigma_\epsilon$.
Recall that $\delta_1,\ldots,\delta_n$ are random vectors in $\R^p$
whose entries $\delta_{i,j}$ are i.i.d. Bernoulli entries with
parameter $\delta$ independent from $(X_1,\cdots,X_n)$ and that the
observations $Y_1,\ldots,Y_n$ satisfy $Y^{(j)}_i =
\delta_{ij}X_{i}^{(j)}$. Denote by $\mathbb P_{\Sigma}$ the
distribution of $(Y_1,\cdots,Y_n)$ and by $\mathbb
P_{\Sigma}^{(\delta)}$ the conditional distribution of
$(Y_1,\cdots,Y_n)$ given $(\delta_{1},\cdots,\delta_n)$. Next, we
note that, for any $1\leq i \leq n$, the conditional random variables
$Y_i\mid \delta_{i}$ are independent Gaussian vectors
$N(0,\Sigma^{(\delta_i)})$ where
\begin{equation}\label{proof-theo2-interm1}
(\Sigma^{(\delta_i)})_{j,k} =
\begin{cases}
\delta_{i,j}\delta_{i,k}\Sigma_{j,k}&\text{if $j\neq k$},\\
\delta_{i,j}\Sigma_{j,j}&\text{otherwise}.
\end{cases}
\end{equation}

Thus, we have $\P_{\Sigma}^{(\delta)} = \otimes_{i=1}^n
\P_{\Sigma^{(\delta_i)}}$. Denote respectively by $\P_\delta$ and $\E_\delta$ the
probability distribution  of $(\delta_1,\cdots,\delta_n)$ and the
associated expectation, and by $\E_{\delta_i}$ the expectation w.r.t $\delta_i$ for any $1\leq i \leq n$. We also denote by $\E_{\Sigma}$
and $\E^{(\delta)}_\Sigma$ the expectation and conditional
expectation associated respectively with $\P_\Sigma$ and
$\P_{\Sigma}^{(\delta)}$.

Next, the Kullback-Leibler divergences $K\big(\P_{{
A_0}},\P_{\Sigma}\big)$ between $\P_{{A_0}}$ and $\P_{\Sigma}$
satisfies
\begin{align}\label{proof-theo2-interm2}
K\left( \P_{{ A_0}},\P_{\Sigma} \right)& = \mathbb E_{A_0}\log
\left(\frac{d\P_{A_0}}{d\P_{\Sigma}}\right) = \mathbb E_{A_0}\log
\left(\frac{d(\P_{\delta}\otimes
\P_{A_0}^{(\delta)})}{d(\P_{\delta}\otimes
\P_{\Sigma}^{(\delta)})}\right) =\E_{\delta}\E_{A_0}^{(\delta)}\log
\left(\frac{d \P_{A_0}^{(\delta)}}{d
\P_{\Sigma}^{(\delta)}}\right)\notag\\
&= \E_{\delta}K\left( \P_{ A_0}^{(\delta)},\P_{\Sigma}^{(\delta)}
\right)= \sum_{i=1}^n \E_{\delta_i} K\left( \P_{
A_0^{(\delta_i)}},\P_{\Sigma^{(\delta_i)}} \right).
\end{align}

Using that $Y_i\mid \delta_i \sim N(0,\Sigma^{(\delta_i)})$ with
$\Sigma^{(\delta_i)}$ defined in (\ref{proof-theo2-interm1}), we get
for any $1\leq i \leq n$, any $\Sigma\in \AA^0$ and any realization
$\delta_i(\omega)\in \{0,1\}^p$ that
\begin{enumerate}
\item $\P_{\Sigma^{(\delta_i(\omega))}}\ll
\P_{A_0^{(\delta_i(\omega))}}$ and hence
$K\big(\P_{{A_0^{(\delta_i(\omega))}}},\P_{\Sigma^{(\delta_i(\omega))}}\big)<\infty$.
\item $\P_{\Sigma^{(\delta_i(\omega))}}$ and $
\P_{A_0^{(\delta_i(\omega))}}$ are supported on a $d_i(\omega)$-dimensional subspace
of $\R^p$ where $d_i= \sum_{j=1}^r \delta_{i,j}\sim
\mathrm{Bin}(r,\delta)$.
\end{enumerate}
Define $J_i = \left\lbrace  j\,:\, \delta_{i,j}=1,\, 1\leq j \leq r
\right\rbrace$. Define the mapping $P_i:\R^p\rightarrow \R^{d_i}$ as
follows $P_i(x)= x_{J_i}$ where for any
$x=(x^{(1)},\cdots,x^{(p)})^\top\in\R^p$, $x_{J_i}\in \R^{d_i}$ is
obtained by keeping only the components $x^{(k)}$ with their index
$k\in J_i$. We denote  by $P_i^{*}\,: \R^{d_i} \rightarrow \R^{p}$
the right inverse of $P_i$.

We note that $P_i A_0^{(\delta_i)}P_i^{*} = I_{d_i}$ and
\begin{align*}
 P_i\Sigma^{(\delta_i)}P_i^{*} &= I_{d_i} + \gamma
\sum_{k=1}^{r-1}\sum_{l=k+1}^{r}\epsilon_{k,l} P_i
B_{k,l}P_i^{*}\1_{k\in J_i}\1_{l\in J_i}\\
&=  I_{d_i} +  W_i.
\end{align*}

Thus we get that
\begin{align*}
K\big(\P_{{A_0^{(\delta_i)}}},\P_{\Sigma^{(\delta_i)}}\big) &= K\big(\P_{I_{d_i}},\P_{ I_{d_i} +  W_i}\big)  \\
&= \frac{1}{2}\mathrm{tr}\left(I_{d_i} +  W_i\right)
-\frac{1}{2}\log\left(\det\left(I_{d_i} +  W_i\right)\right) -
\frac{d_i}{2}.
\end{align*}
Denote by $\lambda_1,\ldots,\lambda_{d_i}$ the eigenvalues of $W_i$.
Note that $|\lambda_j|<1/2$ for any $j=1,\ldots,d_i$ in view of (\ref{spectraltest}) if $a<1/2$. We get,
using the inequality $x-\log(1+x) \leq x^2$ for any $x>-1/2$, that
\begin{align}\label{KLdiv}
K\big(\P_{{A_0^{(\delta_i)}}},\P_{\Sigma^{(\delta_i)}}\big) \ &\leq
\frac{1}{2}\sum_{j=1}^{d_i} \lambda_j^2\notag\\
&\leq \frac{1}{2}\|W_i\|_2^2\notag\\
&\leq \frac{1}{2}\gamma^2
\sum_{k=1}^{r-1}\sum_{l=k+1}^r\|B_{k,l}\|_2^2\1_{k\in
J_i}\1_{l\in J_i}\notag\\
&\leq \gamma^2( d_i^2 - d_i).
\end{align}
Taking the expectation w.r.t. to $\delta_i$ in the above display, we
get for any $1\leq i \leq n$ that
\begin{align*}
\E_{\delta_i} K\left( \P_{
A_0^{(\delta_i)}},\P_{\Sigma^{(\delta_i)}} \right) \leq 
\gamma^2 \E_{\delta_i}\left( d_i^2 - d_i\right)\leq 
\gamma^2 \delta^2 r(r-1),
\end{align*}
since $d_i\sim \mathrm{Bin}(r,\delta)$. Combining the above display
with (\ref{proof-theo2-interm2}), we get
\begin{equation*}
K\left( \P_{{ A_0}},\P_{\Sigma} \right) \leq n \gamma^2
\delta^2r(r-1) = n a^2\frac{1}{\delta^2 n} \delta^2 r(r-1)
\leq a^2r(r-1).
\end{equation*}
Thus, we deduce from the above display that the condition
\begin{equation}\label{eq: condition C}
\frac{1}{\mathrm{Card}(\AA^0)-1}
\sum_{\Sigma\in\AA^0\setminus\{A_0\}}K(\P_{\bf A_0},\P_{\Sigma})\
\leq\ \alpha \log \big(\mathrm{Card}(\AA^0)-1\big)
\end{equation}
is satisfied for any $\alpha>0$ if $a>0$ is chosen as a sufficiently
small numerical constant depending on $\alpha$. In view of
(\ref{lower_2}) and (\ref{eq: condition C}), (\ref{eq:lower1}) now
follows by application of Theorem 2.5 in \cite{tsy_09}.

The lower bound (\ref{eq:lower2}) follows from (\ref{eq:lower1}) by
the following simple argument. Consider the set of matrices
$\mathcal A^0$. For any two distinct matrices $\Sigma_1,\Sigma_2$ of
$\mathcal A^0$, we have
\begin{equation}\label{eq:lower_spectral_1}
  \|\Sigma_1- \Sigma_2\|_\infty\geq  \sqrt{\frac{(1-a)a^2}{16(1+a)^2}}\|\Sigma_{\tilde\epsilon}\|_\infty \sqrt{\frac{\mathbf{r}(\Sigma_{\tilde\epsilon})}{\delta^2
n}},\,\forall \Sigma_{\tilde\epsilon}\in \mathcal A^0.
\end{equation}
Indeed, if (\ref{eq:lower_spectral_1}) does not hold, we get
$$
 \|\Sigma_1- \Sigma_2\|_2^2 <
\frac{(1-a)a^2}{16(1+a)^2}\|\Sigma_{\tilde\epsilon}\|_\infty^2\frac{\mathbf{r}(\Sigma_{\tilde\epsilon})}{\delta^2
n}\mathrm{rank}(\Sigma_{\tilde\epsilon}),\,\forall
\Sigma_{\tilde\epsilon}\in \mathcal A^0,
$$
since $\mathrm{rank}(\Sigma_1- \Sigma_2)\leq r$ by construction of
$\mathcal A^0$. This contradicts (\ref{lower_2}).

Next, (\ref{eq: condition C}) is satisfied for any $\alpha>0$ if $a
>0$ is chosen as a sufficiently small numerical constant depending
on $\alpha$.

Combining (\ref{eq:lower_spectral_1}) with (\ref{eq: condition C})
and Theorem 2.5 in \cite{tsy_09} gives the result.

The case $r=1$ can be treated similarly and is actually easier. Indeed if $\mathbf{r}(\Sigma)=1$, then we have $\mathrm{tr}(\Sigma) = \|\Sigma\|_\infty$ and $\mathrm{rank}(\Sigma) = 1$. Consequently, we can derive the lower bound by testing between the two hypothesis
\begin{equation*}
\Sigma_0 =\left(\begin{array}{c|c}
1&O\\
\hline
O &O
\end{array}\right)\quad \text{and}\quad \Sigma_1 =\left(\begin{array}{c|c}
1+\gamma&O\\
\hline
O &O
\end{array}\right).
\end{equation*}
where $\Sigma_0$ and $\Sigma_1$ are $p\times p$ covariance matrices with only one nonzero component on the first diagonal entry. For these covariance matrices, we have $\mathrm{tr}(\Sigma_0) = \|\Sigma_0\|_\infty = 1$ and $\mathrm{tr}(\Sigma_1) = \|\Sigma_1\|_\infty = 1+\gamma \leq 2$.
Thus we have
$$\|\Sigma_0 -\Sigma_1\|_\infty^{2} =\|\Sigma_0 -\Sigma_1\|_2^{2}\geq \frac{a^2}{\delta^2 n} \geq  c \|\Sigma_i\|_\infty^2 \frac{\mathbf{r}(\Sigma_i)}{\delta^2 n},\quad i=0,1$$
for some absolute constant $c>0$. The rest of the proof is identical
to the case $r\geq 2$.
\end{proof}

\section*{Acknowledgment:} I wish to thank Professor Vladimir
Koltchinskii for suggesting this problem and the observation
(\ref{Sigmarecons}).

\footnotesize{
\bibliographystyle{plain}
\bibliography{panel}
}

\end{document}